\documentclass[reqno]{amsart}

\usepackage{amsmath} 
\usepackage{amsfonts}
\usepackage{amssymb}
\usepackage{amstext}
\usepackage{amsbsy}
\usepackage{amsopn}
\usepackage{amsthm}
\usepackage{amsxtra}
\usepackage{graphicx}
\usepackage{caption}
\usepackage{subcaption}
\usepackage{color}
\usepackage{hyperref}
\usepackage{enumerate}
\usepackage{amscd}

\newtheorem{theorem}{Theorem}[section]
\newtheorem{lemma}[theorem]{Lemma}
\newtheorem{corollary}[theorem]{Corollary}

\newtheorem{conjecture}[theorem]{Conjecture}

\newtheorem*{conjecture*}{Conjecture}
\newtheorem*{claim*}{Claim}
\newtheorem*{theorem*}{Theorem}

\theoremstyle{remark}

\theoremstyle{definition}
\newtheorem{definition}[theorem]{Definition}

\newtheorem{example}[theorem]{Example}

\newcommand{\CA}{\mathcal{A}}

\newcommand{\CL}{\mathcal{L}}

\newcommand{\Z}{\mathbb{Z}}

\newcommand{\N}{\mathbb{N}}

\newcommand{\Id}{{\rm Id}}

\newcommand{\wh}{\widehat}

\newcommand{\rst}[1]{\ensuremath{{\mathbin\upharpoonright}%
\raise-.5ex\hbox{$#1$}}}

\begin{document}

\title{Free ergodic $\Z^2$-systems and complexity}
\author{Van Cyr}
\address{Bucknell University, Lewisburg, PA 17837 USA}
\email{van.cyr@bucknell.edu}
\author{Bryna Kra}
\address{Northwestern University, Evanston, IL 60208 USA}
\email{kra@math.northwestern.edu}

\subjclass[2010]{}

\thanks{The  second author was partially supported by NSF grant 1500670.}




\begin{abstract} 
Using results relating the complexity of a two dimensional subshift to 
its periodicity, we obtain an application to the well-known conjecture of Furstenberg  
on a Borel probability measure on $[0,1)$ which is invariant under both $x\mapsto px \pmod 1$ and $x\mapsto qx \pmod 1$, showing that any potential counterexample has a nontrivial lower bound on its complexity. 
\end{abstract}

\maketitle

\section{Introduction} 

\subsection{Complexity and periodicity}
For a one dimensional symbolic system $(X, \sigma)$, meaning that $X\subset\CA^\Z$ is a closed set, where $\CA$ is a finite alphabet, 
that is closed under the left shift $\sigma\colon\CA^\Z\to\CA^\Z$, the Morse-Hedlund Theorem 
gives a simple relation between the complexity of the system and periodicity.  Namely, if $P_X(n)$ denotes 
the complexity function, which counts the number of nonempty cylinder sets of length $n$ in $X$, then $(X, \sigma)$ 
is periodic if and only if there exists $n\in\N$ such that $P_X(n)\leq n$.  
Both periodicity and complexity have natural generalizations to higher dimensional systems.  For 
example, for a two dimensional system $(X, \sigma, \tau)$, meaning that $X\subset\CA^{\Z^2}$ is a closed set 
that is invariant under the left and down shifts $\sigma, \tau\colon \CA^{\Z^2}\to\CA^{\Z^2}$, 
the two dimensional complexity $P_X(n,k)$ is the number of nonempty $n$ by $k$ cylinder sets.  
In a partial solution to Nivat's Conjecture~\cite{Niv}, the authors~\cite{CK} showed that  if $(X, \sigma, \tau)$ is a transitive $\Z^2$-subshift 
and there exist $n,k\in\N$ such that $P_{X}(n,k)\leq nk/2$, then there exists $(i,j)\in\Z^2\setminus\{(0,0)\}$ such that $\sigma^i\tau^jx=x$ for all $x\in X$. 
In this note, we give an application of this theorem to Furstenberg's well-known ``$\times p, \times q$ problem.'' 

\subsection{The $\times p, \times q$ problem}
Let $S,T\colon[0,1)\to[0,1)$ denote the maps $Sx:=px \pmod 1$ and $Tx:=qx \pmod 1$, 
where $p,q\geq 1$ are multiplicatively independent integers (meaning that $p$ and $q$ are not both powers of the same integer). 
In the 1960's, Furstenberg~\cite{Fur} proved that  any closed subset of $[0,1)$
that is invariant under both $S$ and $T$ is either all of $[0,1)$ or 
is finite.  He asked whether a similar statement holds for measures: 
\begin{conjecture}[Furstenberg] 
Let $\mu$ be a Borel probability measure on $[0,1)$ that is invariant under both $S$ and $T$
and is ergodic for the joint action of $S$ and $T$.  Then either $\mu$ is Lebesgue measure or $\mu$ is atomic. 
\end{conjecture}
 
Progress was made in the 1980's with the work of Lyons~\cite{Lyons}, 
followed soon thereafter by Rudolph's proof that 
positive entropy $h_\mu(\cdot)$ 
of the measure $\mu$ with respect to one of the transformations implies the result for 
relatively prime $p$ and $q$.  This was generalized to multiplicatively independent integers by Johnson: 
\begin{theorem}[Rudolph~\cite{Rud} and Johnson~\cite{johnson}] 
Let $\mu$ be a Borel probability measure on $[0,1)$ that is invariant under both $S$ and $T$ 
and is ergodic for the joint action of $S$ and $T$.  If $h_{\mu}(S)>0$ (or equivalently $h_{\mu}(T)> 0$), 
then $\mu$ is Lebesgue measure. 
\end{theorem} 

One way to interpret this theorem is that the set of $\langle S,T\rangle$-ergodic measures experiences an entropy gap with respect to the one-dimensional action generated by $S$ (or equivalently by $T$).  Informally, if $\mu$ has {\em high entropy} (in this 
case meaning that $h_{\mu}(S)>0$), then  
its entropy with respect to $S$ is actually $\log p$ and $\mu$ is Lebesgue measure.  
Our main theorem is that the set of $\langle S,T\rangle$-ergodic measures also experiences a {\em complexity gap}, 
in a sense we make precise.  We show (Theorem~\ref{thm:main}) that if $\mu$ has {\em low complexity} (meaning that 
a certain function grows subquadratically),  then it actually has {\em bounded complexity} (meaning 
that this function is bounded) and $\mu$ is atomic.  
Moreover, all atomic measures have bounded complexity. 

\subsection{Rephrasing $\times p, \times q$ in symbolic terms}
We begin by recasting Furstenberg's Conjecture and the Rudolph-Johnson Theorem as statements about symbolic dynamical systems.  We start by setting some terminology and notation.  

A {\em (measure preserving) system} $(X, \mathcal{X}, \mu, G)$ is a measure space $X$ with an 
associated $\sigma$-algebra $\mathcal X$, probability 
measure $\mu$, and an abelian group $G$ of measurable, measure preserving transformations. 
If the context is clear, we omit the $\sigma$-algebra from the notation, writing $(X, \mu, G)$, and call it a {\em system}.  
The system $(X, \mu, G)$ is {\em free} if the set $\{x\in X\colon gx=x\}$ has measure 
$0$ for every $g\in G$ and the system is {\em ergodic} if the only sets invariant under the 
action of $G$ have either trivial or full measure.  
It follows that if  $(X, \mu, G)$ is an ergodic system with an abelian group $G$ of 
transformations, then the action of $G$ is free if $g_1^{n_1}\circ\ldots\circ g_k^{n_k}\neq\Id$ 
for any $g_1, \ldots, g_k\in G$ and $(n_1, \ldots, n_k)\neq (0, \ldots, 0)$.

Two systems $(X_1,\mathcal{X}_1,\mu_1,G)$ and $(X_2,\mathcal{X}_2,\mu_2,G)$ are
{\em (measure theoretically) isomorphic} if there exist $X_1^{\prime}\in \mathcal{X}_1$ and 
$X_2^{\prime}\in\mathcal{X}_2$ 
with $\mu_1(X_1') = \mu_2(X_2') = 1$ such that $gX_1'\subset X_1'$ for all $g\in G$ 
and $gX_2'\subset X_2'$ for all $g\in G$, 
and there is an invertible bimeasurable transformation 
$\pi\colon X_1^{\prime}\to X_2^{\prime}$ such that $\pi_*\mu_1=\mu_2$ and 
$\pi g(x)=g\pi(x)$ for all $x\in X_1'$, $g\in G$.  

We are particularly interested in the $\Z^2$-system generated by the
two commuting, measure preserving transformations $S$ and $T$.  In this case, we write 
 $(X, \mathcal X, \mu, S, T)$ for the 
$\Z^2$-system.

A {\em (topological) system} $(X, G)$,  is a compact metric space $X$ and a group $G$ 
of homeomorphisms mapping $X$ to itself. If it is clear from the context that we are referring to a 
topological system, we call $(X,G)$ a {\em system}.  A system is said to be {\em minimal} 
if for any $x\in X$, the orbit $\{gx\colon g\in G\}$ is dense in $X$.  
By the Krylov-Bogolyubov Theorem, every system $(X,G)$ admits an invariant Borel probability measure 
and if this measure is unique, we say that $(X,G)$ is {\em uniquely ergodic}. 
A system $(X,G)$ is {\em strictly ergodic} if it is both minimal and uniquely ergodic.

Let $\CA$ denote a finite alphabet and let
 $\CA^{\Z^2}$ be the set of $\CA$-colorings of $\Z^2$.  
 For $x\in\CA^{\Z^2}$ and $\vec u\in\Z^2$, we denote 
 the element of $\CA$ that $x$ assigns to $\vec u$ by $x(\vec u)$.  With respect to the metric 
$$ 
d(x,y):=2^{-\inf\{|\vec u|\colon x(\vec u)\neq y(\vec u)\}}, 
$$ 
$\CA^{\Z^2}$ is compact and the leftward and downward shift maps $\sigma,\tau\colon\CA^{\Z^2}\to\CA^{\Z^2}$ given by 
\begin{eqnarray} 
\label{eq:left}
(\sigma x)(i,j)&:=&x(i+1,j), \\ 
\label{eq:down}
(\tau x)(i,j)&:=&x(i,j+1) 
\end{eqnarray} 
are homeomorphisms.  A closed set $X\subset\CA^{\Z^2}$ which is invariant under the joint action of $\langle\sigma,\tau\rangle$ is called a {\em $\Z^2$-subshift}. 
(The analogous definitions hold for $\Z^d$-subshifts.)

A uniquely ergodic topological system $(\wh X,\nu, G)$ is said to be a 
{\em topological model} for 
the measure preserving system $(X, \mathcal{X}, \mu, G)$ if there exists a measure 
theoretic isomorphism between $(\wh X,\nu,G)$ and $(X, \mu, G)$.  
Again, we are mainly interested in
 topological systems generated by two transformations, and in this case we denote the topological 
 system by $(\wh X, \sigma, \tau)$.  

The Jewett-Krieger Theorem~\cite{jewett,krieger} states that any ergodic $\Z$-system has 
a strictly ergodic topological model, meaning that the system is measure theoretically 
isomorphic to a minimal, uniquely ergodic topological system.  This was generalized to cover ergodic $\Z^d$-systems by 
Weiss~\cite{weiss},  and further refined by Rosenthal (we only state it for $\Z^2$, as this is the only case relevant for our purposes): 
\begin{theorem}[Rosenthal~\cite{rosenthal}]\label{thm:rosenthal} 
Let $(X,\mathcal{X},\mu,S,T)$ be an ergodic, free $\Z^2$-system with entropy less than $\log k$.  
Then there exists a minimal, uniquely ergodic subshift $\wh X\subset\{1,\dots,k\}^{\Z^2}$ such that if $\sigma,\tau\colon\wh X\to\wh X$ denote the horizontal and vertical shifts (respectively) and if $\nu$ is the unique invariant Borel probability on $\wh X$ and $\mathcal{B}$ denotes the Borel $\sigma$-algera, 
then $(\wh X,\mathcal{B},\nu,\sigma,\tau)$ is a topological model for $(X,\mathcal{X},\mu,S,T)$.
\end{theorem} 
We note that in~\cite{rosenthal}, the proof given shows that $\wh X\subset\{1,\dots,k+1\}^{\Z^2}$ and the result that the shift alphabet can be taken to have only $k$ letters is stated without proof.  However, the size of the alphabet is not relevant for our purposes, other than the fact that it is a finite number.

The subshift $\wh X\subset\{1,\dots,k\}^{\Z^2}$ in the conclusion of Theorem~\ref{thm:rosenthal}
is not uniquely defined, 
and so we make the following definition: 
\begin{definition} 
Let $(X,\mathcal{X},\mu,S,T)$ be an ergodic $\Z^2$-system.  A minimal, uniquely ergodic $\Z^2$-subshift that is measure theoretically isomorphic to $(X,\mathcal{X},\mu,S,T)$ is called a {\em Jewett-Krieger model for $(X,\mathcal{X},\mu,S,T)$}. 
\end{definition} 

Theorem~\ref{thm:rosenthal} guarantees that any  free ergodic $\Z^2$ system of finite entropy has a Jewett-Krieger model.  However the definition is still valid for non-free, ergodic $\Z^2$ systems; the only difference is that Rosenthal's Theorem no longer guarantees that such a model exists.  For the case of interest to us, we show 
(in the proof of Theorem~\ref{thm:main}) that if $\mu$ is $\langle S,T\rangle$-ergodic, 
then either $\mu$ is atomic or the action of $\langle S,T\rangle$ is free.  This motives us to make 
the following observation: a finite, ergodic $\Z^2$-system cannot be free, but it has a Jewett-Krieger model  in a trivial way, 
obtained by partitioning the system into individual points.  

Using this language, 
we can rephrase Furstenberg's Conjecture and the Rudolph-Johnson Theorem 
as equivalent statements about Jewett-Krieger models.  
Fix the transformations $S,T\colon[0,1)\to[0,1)$ to be 
the maps $Sx:=px \pmod 1$ and $Tx:=qx \pmod 1$, 
where $p,q\geq 1$ are multiplicatively independent integers.   By the natural extension, 
we mean the invertible cover (see Section~\ref{sec:natural-extension}).

\begin{conjecture}[Symbolic Furstenberg Conjecture] 
Let $\mu$ be a Borel probability measure on $[0,1)$  with Borel $\sigma$-algebra $\mathcal B$ 
that is invariant under both $S$ and $T$ and ergodic for the joint action.  If 
$\wh X\subset\{0,1\}^{\Z^2}$ is a Jewett-Krieger model for the natural extension of $([0,1),\mathcal{B},\mu,S,T)$, 
then either $\wh X$ is finite or $\mu$ is Lebesgue measure. 
\end{conjecture} 
\begin{theorem}[Symbolic Rudolph-Johnson Theorem]\label{Rud-symbolic} 
Let $\mu$ be a Borel probability measure on $[0,1)$ with Borel $\sigma$-algebra $\mathcal B$ 
 which is invariant under both $S$ and $T$ and is ergodic for the joint action.  Let $\wh X\subset\{0,1\}^{\Z^2}$ be a Jewett-Krieger model for the natural extension of $([0,1),\mathcal{B},\mu,S,T)$ and let $\sigma,\tau\colon\wh X\to\wh X$ denote the horizontal and vertical shifts (respectively).  If either $h_{\nu}(\sigma)>0$ or $h_{\nu}(\tau)>0$, then $\mu$ is Lebesgue measure. 
\end{theorem} 
\begin{proof} 
An isomorphism of the $\Z^2$-systems $(X,\mathcal{X},\mu,S,T)$ and $(\wh X,\mathcal{B},\nu,\sigma,\tau)$ restricts to an isomorphism of the $\Z$-systems $(X,\mathcal{X},\mu,S)$ and $(\wh X,\mathcal{B},\nu,\sigma)$, 
and so $h_{\mu}(S)=h_{\nu}(\sigma)$.  Similarly $h_{\mu}(T)=h_{\nu}(\tau)$.  
The statement then follows immediately from the Rudolph-Johnson Theorem. 
\end{proof} 

\subsection{Combinatorial rephrasing of measure theoretic entropy} 
The appeal of Theorem~\ref{Rud-symbolic} is that 
the hypothesis that $h_{\nu}(\sigma)>0$ (or equivalently that $h_{\nu}(\tau)>0$) can be phrased purely as a combinatorial statement about the frequency with which words in the language of $\wh X$ occur in larger words in the language of $\wh X$.  
To explain this, we start with some definitions. 

If $X\subset\CA^\Z$ is a subshift over the finite alphabet $\CA$, we write $x = (x(n)\colon n\in\Z)$.  
A  {\em word} is a defined to be a finite sequence of symbols 
contained consecutively in some $x$ and we let $|w|$ denote the number of symbols in $w$ (it may be finite or infinite).  
A word $w$ is a {\em subword} of a word $u$ if the symbols in the word $w$ occur somewhere in $u$ as consecutive 
symbols.  
The {\em language $\CL = \CL(X)$} of $X$ is defined to be the collection 
of all finite subwords that arise in elements of $X$.
If $w\in\CL(X)$, let $[w]$ denote the {\em cylinder set} it determines, meaning that 
$$
[w] = \{u\in\CL\colon u(n) = w(n) \text{ for } 1\leq n\leq |w|\}.
$$

These definitions naturally generalize to a two dimensional subshift $X\subset\CA^{\Z^2}$, 
and for $x\in\CA^{\Z^2}$ we write $x = (x(\vec u)\colon\vec u\in\Z^2)$.  A {\em word} is a 
finite, two dimensional configuration that is convex and connected (as a subset of $\Z^2$), and a {\em subword}
is a configuration contained in another word. 
If $F\subset\Z^2$ is finite and $\beta\in\CA^F$, then the {\em cylinder set of shape $F$ determined by $\beta$} is 
defined to be the set 
$$ 
[F;\beta]:=\{x\in\CA^{\Z^2}\colon x(\vec u)=\beta(\vec u)\text{ for all }\vec u\in F\}.
$$

\begin{lemma} 
\label{lemma:compute}
Let $(\wh X,\mathcal{B},\nu,\sigma,\tau)$ be a strictly ergodic $\Z^2$-subshift.  Let $w$ be a $(2n+1)\times (2n+1)$ word in the language of $\wh X$ and let $[w]$ denote the cylinder set determined by placing the word $w$ 
centered at $(0,0)$.  
Let  $u_1,u_2,u_3,\dots$ be words in the language of $\wh X$ such that $u_i$ is a square of size
$(2n+2i+1)\times(2n+2i+1)$.  
If $N(w,u_i)$ denotes the number of times $w$ occurs as a subword of $u_i$, then 
$$\nu[w]=\lim_{i\to\infty}N(w,u_i)/(2i+1)^2.$$
\end{lemma} 
\begin{proof} 
By unique ergodicity, the Birkhoff averages of a continuous function converge uniformly to the integral of the function.  In particular, this applies to the continuous function $1_{[w]}$, so the limit exists and is independent of the sequence $\{u_i\}_{i=1}^{\infty}$. 
\end{proof} 

For $m,n\in\N$, let $\mathcal{P}(m,n)$ be the partition of $\wh X$ according to cylinder sets of shape $[0,m-1]\times[-n+1,n-1]$.  Observe that (recall that $\sigma$, as defined in~\eqref{eq:left}, denotes the left shift)
$$ 
\mathcal{P}(m,n)=\bigvee_{i=0}^k\sigma^{-i}\mathcal{P}(1,n) 
$$ 
and that $\bigvee_{i=-k}^k\sigma^i\mathcal{P}(1,n)$ is the partition of $\wh X$ into symmetric $(2m+1)\times(2n+1)$-cylinders centered at the origin.  Therefore, $\{\mathcal{P}(1,n)\}_{n=1}^{\infty}$ is a sufficient (in the sense of Definition 4.3.11 in~\cite{KH}) family of partitions to 
generate the Borel $\sigma$-algebra of the system $(\wh X,\mathcal{B},\nu,\sigma)$, where we view this as a $\Z$-system 
with respect to the horizontal shift $\sigma$.  
Let $h_\nu(\sigma,\mathcal Q)$ denote the measure theoretic entropy of the system $(\wh X,\mathcal{B}, \nu, \sigma)$ with respect to the partition $\mathcal Q$ and let $h_\nu(\sigma)$ denote the measure theoretic entropy of the system.  
It follows that 
\begin{eqnarray*} 
h_{\nu}(\sigma)&=&\sup_n h_{\nu}(\sigma,\mathcal{P}(1,n)) \\ 
&=&\lim_{n\to\infty}h_{\nu}\left(\sigma,\mathcal{P}(1,n)\right) \\ 
&=&-\lim_{n\to\infty}\lim_{m\to\infty}\frac{1}{m}\sum_{w\in\mathcal{P}(m,n)}\nu[w]\log\nu[w] \\ 
&=&-\lim_{n\to\infty}\lim_{m\to\infty}\lim_{i\to\infty}\frac{1}{m}\sum_{w\in\mathcal{P}(m,n)}\frac{N(w,u_i)}{(2i+1)^2}\cdot\log\frac{N(w,u_i)}{(2i+1)^2}
\end{eqnarray*} 
by Lemma~\ref{lemma:compute}.  
In other words, the Rudolph-Johnson Theorem is equivalent to: 
\begin{theorem}[Combinatorial Rudolph-Johnson Theorem]\label{Rud:comb} 
Let $\mu$ be a Borel probability measure on $[0,1)$ with Borel $\sigma$-algebra $\mathcal B$
and assume that $\mu$ is invariant under both $S$ and $T$, and ergodic for the joint action.  Let $\wh X$ be a Jewett-Krieger model for the natural extension of $([0,1),\mathcal{B},\mu,S,T)$ and without loss suppose the horizontal shift on $\wh X$ is intertwined with $S$ under this isomoprhism.  If  
$$ 
-\lim_{n\to\infty}\lim_{m\to\infty}\lim_{i\to\infty}\frac{1}{m}\sum_{w\in\mathcal{P}(m,n)}\frac{N(w,u_i)}{(2i+1)^2}\cdot\log\frac{N(w,u_i)}{(2i+1)^2}>0, 
$$ 
then the value of this limit is $\log p$ and $\mu$ is Lebesgue measure. 
\end{theorem} 

\subsection{Complexity of subshifts} 

If $X\subset\CA^{\Z^2}$ is a nonempty subshift, then its {\em complexity function} is the function $P_X\colon\{\text{finite subsets of $\Z^2$}\}\to\N$ given by 
$$ 
P_X(F):=\bigl\vert\{\beta\in\CA^F\colon[F;\beta]\cap X\neq\emptyset\}\bigr\vert. 
$$ 
Let $R_n:=\{(i,j)\in\Z^2\colon1\leq i,j\leq n\}$ denote the $n\times n$ rectangle in $\Z^2$.  A standard notion of the complexity of a subshift $X\subset\CA^{Z^2}$ is the asymptotic growth rate of $P_X(R_n)$.  Observe that $P_X(R_n)$ is  bounded (in $n$) if and only if $X$ is finite.  Moreover, $P_X(R_n)$ grows exponentially if and only if $(X,\sigma,\tau)$ has positive topological entropy. 

We are now in a position to state our main technical result.   
\begin{theorem}\label{thm:main} 
Let $\mu$ be a Borel probability measure on $[0,1)$ with Borel $\sigma$-algebra $\mathcal B$.  Assume 
that $\mu$ is invariant under both $S$ and $T$ and
ergodic for the joint action, and let $\wh X\subset\{0,1\}^{\Z^2}$ be a Jewett-Krieger model for the natural extension of 
$([0,1),\mathcal{B},\mu,S,T)$.  If there exists $n\in\N$ such that $P_{\wh X}(R_n)\leq\frac{1}{2} n^2$, 
then $P_{\wh X}(R_n)$ is bounded (independent of $n$) and $\wh X$ is finite.  In particular, $\mu$ is atomic. 
\end{theorem} 

This gives a nontrivial complexity gap 
for the set of $\langle S,T\rangle$-ergodic probability measures, which is our main result: 
\begin{corollary}[Complexity gap]\label{cor:gap} 
Let $\mu$ be a Borel probability measure on $[0,1)$ which is invariant under both $S$ and $T$ and ergodic for the joint action, and let $\wh X\subset\{0,1\}^{\Z^2}$ be a Jewett-Krieger model for the natural extension of $([0,1),\mathcal{B},\mu,S,T)$.  Then either $P_{\wh X}(R_n)$ is bounded (and $\mu$ is atomic) or 
$$ 
\liminf_{n\to\infty}\frac{P_{\wh X}(R_n)}{n^2}\geq\frac{1}{2}. 
$$ 
\end{corollary} 

\noindent This gap is nontrivial in the following sense: there exist aperiodic, strictly ergodic $\Z^2$-subshifts whose complexity function grows subquadratically.  The statement made by Corollary~\ref{cor:gap} is that any such subshift cannot be a Jewett-Krieger model of any $\langle\times p,\times q\rangle$-ergodic measure on $[0,1)$. 

\begin{example} 
Let $X\subset\{0,1\}^{\Z}$ be a Sturmian shift (see~\cite{MH} for the definition).  Then $X$ is strictly ergodic and $P_X(n)=n+1$ for all $n\in\N$.  Let $Y\subset\{0,1\}^{\Z^2}$ be the subshift whose points are obtained by placing each $x\in X$ along the $x$-axis in $\Z^2$ and then copying vertically (i.e. each point in $Y$ is vertically constant and its restriction to the $x$-axis is an element of $X$).  It follows that $Y$ is strictly ergodic and that $P_Y(R_n)=n+1$ for all $n\in\N$.  Corollary~\ref{cor:gap}  shows that $Y$ is not a Jewett-Krieger model for any $\langle\times p,\times q\rangle$-ergodic measure on $[0,1)$. 
\end{example} 

\subsection{Remarks on complexity growth} 
We conclude our introduction with a few brief remarks on Theorem~\ref{thm:main} and Corollary~\ref{cor:gap}.  We 
show (Lemma~\ref{lemma:finite}) that any Jewett-Krieger model $\wh X$ for an atomic $\langle S,T\rangle$-ergodic measure is a strictly ergodic $\Z^2$-subshift containing only doubly periodic $\Z^2$-colorings, meaning that 
there are only finitely many points in $\wh X$.  
From this, it is easy to deduce 
that $P_{\wh X}(R_n)$ is bounded independently  of $n$ (by the number of points in $\wh X$).  
Moreover, we show that if $\wh X$ is a Jewett-Krieger model for $\mu$ 
and if $\wh X$ contains only doubly periodic $\Z^2$-colorings,  then $\mu$ is atomic. 

A strategy for proving Theorem~\ref{thm:main} is therefore to find a nontrivial growth rate of 
$P_{\wh X}(R_n)$ which implies that $\wh X$ contains only doubly periodic $\Z^2$-colorings.  A simple example of such a rate follows from the classical Morse-Hedlund Theorem~\cite{MH}: if there exists $n\in\N$ such that $P_{\wh X}(R_n)\leq n$, then $\wh X$ contains only doubly periodic $\Z^2$-colorings (see e.g. the proof of Theorem 1.2 in~\cite{PO}).  In fact this bound is sharp: there exist $\Z^2$-colorings that are not doubly periodic and yet
satisfy $P_{\wh X}(R_n)=n+1$ for all $n\in\N$.  
Many other subquadratic growth rates can also be realized by strictly ergodic $\Z^2$-subshifts that do not contain doubly periodic points (see, for example,~\cite{pansiot}).  
Therefore, a weak version of Theorem~\ref{thm:main} that replaces the assumption that there exists $n\in\N$ such that $P_{\wh X}(R_n)\leq\frac{1}{2}\cdot n^2$ with the stronger assumption that there exists $n\in\N$ such that $P_{\wh X}(R_n)\leq n$, follows from the Morse-Hedlund Theorem.  
However, this weak theorem relies on the fact that there are simply no strictly ergodic $\Z^2$-subshifts for which $P_{\wh X}(R_n)$ is unbounded but for which $P_{\wh X}(R_n)\leq n$ (for some $n$).  
The complexity gap provided by this weak theorem is therefore trivial in the sense that 
there are no strictly ergodic $\Z^2$-subshifts whose complexity function lies in this gap.  

On the other hand, there do exist strictly ergodic $\Z^2$-subshifts with unbounded complexity 
and  such that $P_{\wh X}(R_n)<\frac{1}{2}\cdot n^2$.  This is the interest in Theorem~\ref{thm:main} and Corollary~\ref{cor:gap}.  
The content of the theorem is that although such $\Z^2$-systems exist, they can not be 
Jewett-Krieger models of $\langle S,T\rangle$-ergodic measures on $[0,1)$.  
This is analogous to Theorem~\ref{Rud:comb}, 
which says that although there are strictly ergodic $\Z^2$-subshifts that have small but positive entropy, they are not Jewett-Krieger models of $\langle S,T\rangle$-ergodic measures on $[0,1)$.  Moreover, analogous to the hypothesis of Corollary~\ref{cor:gap} which relies on the growth rate of $P_{\wh X}(\cdot)$, the hypothesis of Theorem~\ref{Rud:comb} is a condition on the growth rate of the relative complexity function $N(\cdot,\cdot)$ of Lemma~\ref{lemma:compute}, with respect to the action of the horizontal shift (a similar statement holds for the vertical shift). 

\section{Proof of Theorem~\ref{thm:main}}

Throughout this section, we assume that $p, q\geq 2$ are multiplicatively independent 
integers and that $\mu$ is a Borel probability measure on $[0,1)$ which is invariant under both 
\begin{eqnarray*} 
Sx&:=&px\text{ (mod $1$)}; \\ 
Tx&:=&qx\text{ (mod $1$)} 
\end{eqnarray*} 
and is ergodic with respect to the joint action $\langle S,T\rangle$. 
Let $\mathcal{B}$ denote  the associated Borel $\sigma$-algebra on $[0,1)$
%

\subsection{The natural extension}
\label{sec:natural-extension}
Let $X$ be the natural extension of the $\N^2$-system $([0,1),\mathcal{B},\mu,S,T)$.  Specifically (following~\cite{Pet}), let 
$$ 
X:=\left\{y\in[0,1)^{\Z^2}\colon y(i+1,j)=Sy(i,j)\text{ and }y(i,j+1)=Ty(i,j)\text{ for all }i,j\in\Z\right\}, 
$$ 
and for $(i,j)\in\Z^2$ let $\pi_{(i,j)}\colon X\to[0,1)$ be the map $\pi_{(i,j)}(y)=y(i,j)$.  Define a countably additive measure $\mu_X$ on the $\sigma$-algebra 
$$ 
\bigcup_{i=0}^{\infty}\pi_{(-i,-i)}^{-1}\mathcal{B} 
$$ 
by setting $\mu_X(\pi_{(-i,-i)}^{-1}A):=\mu(A)$.  Let $\mathcal{X}$ be the completion of this $\sigma$-algebra with respect to $\mu_X$.  Let $S_X,T_X\colon X\to X$ be the left shift and the down shift, respectively.  Thus $\pi_{(0,0)}$ defines a measure 
theoretic factor map from $(X,\mathcal{X},\mu_X,S_X,T_X)$ to $([0,1),\mathcal{B},\mu,S,T)$.  Moreover, $\mu_X$ is ergodic if and only if $\mu$ is ergodic.  By construction, $h_{\mu}(S)=h_{\mu_X}(S_X)$, $h_{\mu}(T)=h_{\mu_X}(T_X)$, and $h_{\mu}(\langle S,T\rangle)=h_{\mu_X}(\langle S_X, T_X\rangle)$.  

The advantage of working with $(X,\mathcal{X},\mu_X,S_X,T_X)$ instead of the original system is that the natural extension is an ergodic $\Z^2$-system. 

\subsection{Jewett-Krieger models and periodicity} 

If the two dimensional entropy of a system is positive, 
then the entropy of every one dimensional subsystem is infinite
(for a proof, see, for example,~\cite{Sinai}).
In our setting, since $h_{\mu}(S)\leq h_{\text{top}}(S)=\log(p)$
 (and $h_{\mu}(T)\leq h_{\text{top}}(T)=\log(q)$), 
it follows that the measure theoretic entropy of the joint action generated by 
$\langle S,T\rangle$ on $[0,1)$ with respect to $\mu$ is also zero.  
It follows that the measure theoretic entropy with respect to $\mu_X$ of the joint action on $X$  generated by $\langle S_X,T_X\rangle$ is zero.  
Therefore, by Theorem~\ref{thm:rosenthal}, there exists a strictly ergodic subshift $\wh X\subset\{0,1\}^{\Z^2}$ such that $(X,\mathcal{X},\mu_X,S_X,T_X)$ is measure theoretically isomorphic to $(\wh X,\wh{\mathcal{X}},\nu,\sigma,\tau)$, 
where $\wh{\mathcal{X}}$ is the Borel $\sigma$-algebra on $\wh X$, 
$\sigma,\tau\colon\wh X\to\wh X$ denote the left shift and down shift (respectively), and $\nu$ is the unique $\langle\sigma,\tau\rangle$-invariant Borel probability measure.  Note that the choice of $\wh X$ is not necessarily unique.  

%
%

\begin{lemma}
\label{lemma:finite}  
If $(X,\mathcal{X},\mu_X,S_X,T_X)$ is an atomic system, then 
any Jewett-Krieger model $(\wh X,\wh{\mathcal{X}},\nu,\sigma,\tau)$ for $(X,\mathcal{X},\mu_X,S_X,T_X)$ is finite. 
\end{lemma}
\begin{proof} 
Let $\pi\colon (\wh X,\wh{\mathcal{X}},\nu,\sigma,\tau)\to (X,\mathcal{X},\mu_X,S_X,T_X)$ be an isomorphism and let $x\in X$ be an atom.  Then there exist full measure sets $\wh X_1\subset\wh X$ and $X_1\subset X$ such that $\pi\colon\wh X_1\to X_1$ is a bijection which interwines the $\Z^2$ actions.  Every atom in $X$ is contained in $X_1$, 
and if $x\in X_1$ is an atom then there exists unique $y\in\wh X_1$ such that $\pi(y)=x$.  It follows that $\nu(\{y\})=\mu_X(\{x\})>0$ and so $y$ is an atom in $\wh X$.  By the Poincar\'e Recurrence Theorem, there exists $(i,j)\in\Z^2\setminus\{(0,0)\}$ such that $S_X^iT_X^jy=y$.  Let $V_y:=\{(i,j)\in\Z^2\setminus\{(0,0)\}\colon S_X^iT_X^jy=y\}$ 
be the (nonempty) set of nontrivial period vectors for $y$.  If $\text{dim}(\text{Span}(V_y))=1$, then 
$$ 
\lim_{N\to\infty}\frac{1}{(2N+1)^2}\sum_{-N\leq i,j\leq N}1_{\{y\}}(S_X^iT_X^jy)=0<\nu(\{y\}), 
$$ 
which contradicts the pointwise ergodic theorem.  Therefore $\text{dim}(\text{Span}(V_y))=2$ and $y\in\CA^{\Z^2}$ is doubly periodic.  Moreover, for $\nu$-a.e. $z\in\wh X$ we have $S_X^iT_X^jz=y$ for some $(i,j)\in\Z^2$ and so $z$ is also doubly periodic (with periods equal to those of $y$).  Thus there are only finitely many points $z\in\wh X$. 
\end{proof}

Since $\wh X$ is minimal, and hence transitive, we can use the following tool for 
studying the dynamics of $(X,\mathcal{X},\mu_X,S_X,T_X)$: 
\begin{theorem}[Cyr \& Kra~\cite{CK}]\label{thm:CK} 
If $(X, \sigma, \tau)$ is a transitive $\Z^2$-subshift 
and there exist $n,k\in\N$ such that $P_{X}(n,k)\leq nk/2$, then there exists $(i,j)\in\Z^2\setminus\{(0,0)\}$ such that $\sigma^i\tau^jx=x$ for all $x\in X$. 
\end{theorem}

\begin{lemma}\label{lem:intermediate} 
If there exists $(i,j)\in\Z^2\setminus\{(0,0\}$ such that $\sigma^i\tau^j x  = x$ for every  
$x\in \wh X$, then $S_X^iT_X^j x=x$ for $\mu$-almost every $x\in X$. 
\end{lemma} 

\begin{proof} 
Let $\psi\colon\wh X\to X$ be an isomorphism.  Thus there exist $\wh X_1\subset\wh X$ and $X_1\subset X$ such that $\nu(\wh X_1)=\mu_X(X_1)=1$, $\psi\colon\wh X_1\to X_1$ is a bi-measurable bijection, $\psi_*\nu=\mu_X$, $\psi\circ\sigma=S_X\circ\psi$, and 
$\psi\circ\tau=T_X\circ\psi$.  Let $E=\{x\in X_1\colon S_X^iT_X^jx\neq x\}$.  
Since $\psi^{-1}(E)=\{y\in\wh X_1\colon\sigma^i\tau^jy\neq y\}$, 
it follows that $\mu_X(E)=\nu(\psi^{-1}(E))=0$.  
\end{proof}

\begin{theorem}
\label{th:one-to-all}
If there exist $n,k\in\N$ such that $P_{\wh X}(n,k)\leq nk/2$, then $\mu$ is atomic.  Moreover, if $\wh Y$ is any other Jewett-Krieger model for $([0,1),\mathcal{B},\mu,S,T)$, then $P_{\wh Y}(n,k)$ is bounded independent of $n,k\in\N$. 
\end{theorem} 
\begin{proof} 
Combining Theorem~\ref{thm:CK} and Lemma~\ref{lem:intermediate}, there exist $(i,j)\in\Z^2\setminus\{(0,0\}$ such that 
$S_X^iT_X^jx=x$ for $\mu_X$-a.e. $x\in X$.  Therefore $(S_X^iT_X^jx)(0,0)=x(0,0)$ for $\mu_X$-a.e. $x\in X$.  It is immediate that we also have $(S_X^{-i}T_X^{-j}x)(0,0)=x(0,0)$ for $\mu_X$-a.e. $x\in X$.  So there are two cases to consider, depending on the the sign of $i\cdot j$.

\subsubsection*{Case 1} Suppose $i\cdot j\geq0$.  Then, replacing by $-i$ and $-j$ if necessary, 
we can assume that both $i$ and $j$ are nonnegative.  Set
$E:=\{y\in[0,1)\colon S^iT^jy\neq y\}$ and let $y\in E$.  Then if $x\in\pi^{-1}(y)$,  we 
have that $S_X^iT_X^jx\neq x$.  Thus 
$\mu(E)=\mu_X(\pi^{-1}(E))=0$ and so $S^iT^jy=y$ for $\mu$-a.e. $y\in[0,1)$. 

Now observe that $S^iT^jy=y$ is equivalent to the statement that 
$$ 
p^iq^jy=y\text{ (mod $1$)},
$$ 
which only has finitely many solutions in the interval $[0,1)$.  Therefore, $\mu$ is supported on a finite set.  
Since $\mu$ is $\langle S,T\rangle$-invariant, this set must be $S$- and $T$-invariant.  
Therefore there exist $a,b\in\N$ such that $S^a$ and $T^b$ are both equal to the identity 
$\mu$-almost everywhere.

\subsubsection*{Case 2} Suppose $i\cdot j<0$.  Again, replacing by $-i$ and $-j$ if necessary, 
we can assume that $i<0$ and $j>0$.  Now set $E:=\{y\in[0,1)\colon S^{|i|}y\neq T^jy\}$.  Thus if $y\in E$ and $x\in\pi^{-1}(y)$, then $x(-i,j)\neq x(0,0)=y$ as $S^{|i|}(x(-i,j))=x(0,j)=T^j(x(0,0))$ by construction.  Therefore $S^iT^jx\neq x$ and so $\mu(E)=\mu_X(\pi^{-1}(E))=0$.  It follows that $S^{|i|}y=T^jy$ for $\mu$-a.e. $y\in[0,1)$. 

Finally observe that $S^{|i|}y=T^jy$ is equivalent to 
$$ 
p^{|i|}y=q^jy\text{ (mod $1$)}. 
$$ 
As $p$ and $q$ are multiplicatively independent, there are only finitely many solutions in the interval $[0,1)$.  Therefore, again, $\mu$ is supported on a finite set and there exist $a,b\in\N$ such that $S^a$ and $T^b$ are  both equal to the identity 
$\mu$-almost everywhere.

This establishes the first claim of the theorem.   By Lemma~\ref{lemma:finite}, 
any Jewett-Krieger model of an atomic system is finite, and the second statement follows.  
\end{proof}

We use this to complete the proof of Theorem~\ref{thm:main}:
\begin{proof}[Proof of Theorem~\ref{thm:main}]
Let $\mu$ be a Borel probability measure on $[0,1)$ that is $\langle S,T\rangle$-ergodic.  
If this two dimensional action is not free, arguing 
as in the proof of Theorem~\ref{th:one-to-all} that $\mu$ is an atomic measure, we are done.  Thus we can assume that the action is free, and similarly the action for the natural extension is also free. 

Let $(\wh X,\wh{\mathcal{X}},\nu,\sigma,\tau)$ be a Jewett-Krieger model for 
 the natural extension of the system $([0,1),\mathcal{B},\mu,S,T)$.  
 If there is no such model satisfying the 
 additional property that there exist $n,k\in\N$ satisfying $P_{\wh X}(n,k)\leq nk/2$, then the conclusion of the Theorem holds vacuously.  Thus it suffices to assume that there exists a Jewett-Krieger model $(\wh X,\wh{\mathcal{X}},\nu,\sigma,\tau)$ with the property that there exist $n,k\in\N$ satisfying $P_{\wh X}(n,k)\leq nk/2$.   
By Theorem~\ref{th:one-to-all},  $([0,1),\mathcal{B},\mu,S,T)$ is atomic.  
\end{proof}

\section{Higher dimensions}
Theorem~\ref{thm:main} shows that if $\mu$ is any nonatomic $\times p$, $\times q$ ergodic measure then the natural extension of $([0,1),\mathcal{X},\mu,S,T)$ cannot be measurably isomorphic to any $\Z^2$-subshift of whose complexity function satisfies $P_X(n,n)=o(n^2)$.  It is natural to ask whether this result can be generalized to higher dimensions.  In particular, if $p_1,\dots,p_d$ are a multiplicatively independent set of integers and $\mu$ is a nonatomic $\times p_1,\dots,\times p_d$ ergodic measure, we can ask if the natural extension of $(X,\mathcal{X},\mu,\times p_1,\dots,\times p_d)$ could have a topological model whose complexity function is $o(n^d)$. 

The same method used in the two dimensional case suggests a path to proving this result.  If one could show that any free,  strictly ergodic $\Z^d$-subshift whose complexity function is $o(n^d)$ is periodic, then it would follow that no such topological model for $\mu$ exists.  However, the analog of Theorem~\ref{thm:CK} in dimension $d>2$ is false.  Julien Cassaigne~\cite{Cas} has shown that for $d>2$, there exists a minimal $\Z^d$-subshift $X$ whose elements are not periodic in any direction, and is such that for any $\varepsilon>0$ we have $P_X(n,n,\dots,n)=o(n^{2+\varepsilon})$.  On the other hand, the authors have recently shown~\cite{CK2} that the analog of Theorem~\ref{thm:CK} does hold for dimension $d>2$ if a certain expansiveness assumption is imposed on the subshift. 

If $Y\subset\CA^{\Z^d}$ is a subshift, then we say that the $x$-axis in $\Z^d$ is {\em strongly expansive} if whenever $x,y\in X$ have the same restriction to the $x$-axis, we have $x=y$.  In this case, if $X\subset\CA^{\Z}$ is the subshift obtained by restricting elements of $Y$ to the $x$-axis, then there exist homeomorphisms $\tau_1,\dots,\tau_{d-1}\colon X\to X$ which commute pairwise and with the shift 
$\sigma$ and are such that for any $y\in Y$ we have $y(i_1,i_2,\dots,i_d)=\left(\tau_1^{i_1}\tau_2^{i_2}\cdots\tau_{d-1}^{i_{d-1}}\sigma^{i_d}\pi_X(y)\right)(0)$ for all $i_1,\dots,i_d\in\Z^d$, 
where $\pi_X(y)$ denotes the restriction of $y$ to the $x$-axis.  In previous work, we have shown that: 
\begin{theorem}[Cyr \& Kra~\cite{CK2}] 
Let $X\subset\CA^{\Z}$ be a minimal subshift and let $\tau_1,\dots,\tau_{d-1}\colon X\to X$ be homeomorphisms of $X$ that commute with the shift $\sigma$.  If $\langle\sigma,\tau_1,\dots,\tau_{d-1}\rangle\cong\Z^d$, then $\liminf_{n\to\infty}P_X(n)/n^d>0$. 
\end{theorem} 
\noindent With some additional effort, the same result can be shown if the assumption that $(X,\sigma)$ is minimal (as a $\Z$-system) is relaxed to only require that $(X,\sigma,\tau_1,\dots,\tau_d)$ is minimal (as a $\Z^d$-system).  Thus, the only obstruction to generalizing Theorem~\ref{thm:main} to the higher dimensional setting is the following: 

\begin{conjecture} 
For every nonatomic Borel probability $\mu$ on $[0,1)$ which is ergodic for the joint action of $\times p_1,\dots,\times p_d$, there is a strongly expansive, minimal topological model for $(X,\mathcal{X},\mu,\times p_1,\dots,\times p_d)$. 
\end{conjecture} 

\noindent If this conjecture holds, then if follows that any such system is measurably isomorphic to a subshift whose complexity function grows on the order of $n^d$. 



\end{document}